\documentclass[revision]{FPSAC2024}

\usepackage{tikz-cd}
\usepackage{tikz, tikz-3dplot, pgfplots}
\usepackage{bm, relsize}
\usepackage{tkz-graph}
\usetikzlibrary[positioning,patterns] 

\newtheorem{theorem}{Theorem}

\newtheorem{defn}{Definition}
\newtheorem{lem}{Lemma}
\newtheorem{proposition}{Proposition}

\newtheorem{remark}{Remark}

\newcommand{\la}{\lambda}

\newcommand{\CC}{\mathbb{C}} 
\newcommand{\PP}{\mathbb{P}} 
\newcommand{\ZZ}{\mathbb{Z}} 
\newcommand{\id}{\operatorname{id}} 
\newcommand{\kk}{\mathbf{k}} 
\newcommand{\QSym}{\operatorname{QSym}}

\newcommand{\SYT}{\operatorname{SYT}}
\newcommand{\SSYT}{\operatorname{SSYT}}
\newcommand{\std}{\operatorname{std}}

\newcommand{\nega}{\operatorname{neg}}
\newcommand{\Des}{\operatorname{Des}}

\newcommand{\Peak}{\operatorname{Peak}}

\makeatletter
\providecommand*{\shuffle}{%
  \mathbin{\mathpalette\shuffle@{}}%
}
\newcommand*{\shuffle@}[2]{%
  \sbox0{$#1\vcenter{}$}%
  \kern .15\ht0 
  \rlap{\vrule height .25\ht0 depth 0pt width 2.5\ht0}%
  \raise.1\ht0\hbox to 2.5\ht0{%
    \vrule height 1.75\ht0 depth -.1\ht0 width .17\ht0 %
    \hfill
    \vrule height 1.75\ht0 depth -.1\ht0 width .17\ht0 %
    \hfill
    \vrule height 1.75\ht0 depth -.1\ht0 width .17\ht0 %
  }%
  \kern .15\ht0 
}
\makeatother

\tikzstyle{bsq}=[rectangle, draw, thick, minimum width=1cm, minimum height=1cm] 
\tikzstyle{bver}=[rectangle, draw, thick, minimum width=1cm, minimum height=2cm]
\tikzstyle{bhor}=[rectangle, draw, thick, minimum width=2cm, minimum height=1cm]

\tikzstyle{bsqg}=[rectangle, draw=gray!30!white, thick, fill=gray!30!white, minimum width=1cm, minimum height=1cm] 
\tikzstyle{bverg}=[rectangle, draw=gray!30!white, thick, fill=gray!30!white, minimum width=1cm, minimum height=2cm]
\tikzstyle{bhorg}=[rectangle, draw=gray!30!white, thick, fill=gray!30!white, minimum width=2cm, minimum height=1cm]

\title{Quasisymmetric expansion of Hall-Littlewood symmetric functions}

\author[D. Grinberg \and E.A. Vassilieva]{Darij Grinberg\thanks{\href{mailto:darijgrinberg@gmail.com}{darijgrinberg@gmail.com}}\addressmark{1}, \and Ekaterina A. Vassilieva\thanks{\href{mailto:katya@lix.polytechnique.fr}{katya@lix.polytechnique.fr}}\addressmark{2}}

\address{\addressmark{1}Department of Mathematics, Drexel University, Philadelphia, PA 19104, USA \\ \addressmark{2}LIX, Ecole Polytechnique, Palaiseau, France}

\received{\today}

\abstract{In our previous works we introduced a $q$-deformation of the generating functions for enriched $P$-partitions. We call the evaluation of this generating functions on labelled chains, the $q$-fundamental quasisymmetric functions. These functions interpolate between Gessel's fundamental ($q=0$) and Stembridge's peak ($q=1$) functions, the natural quasisymmetric expansions of Schur and Schur's $Q$-symmetric functions. In this paper, we show that our $q$-fundamental functions provide a quasisymmetric expansion of Hall-Littlewood $S$-symmetric functions with parameter $t=-q$.}

\resume{Dans nos travaux précédents, nous avons introduit une $q$-déformation des fonctions génératrices pour les $P$-partitions  enrichies. Nous nommons l'évaluation de ces fonctions génératrices sur les chaînes étiquetées, les fonctions quasisymétriques $q$-fondamentales. Ces fonctions interpolent entre les fonctions fondamentales de Gessel ($q=0$) et les fonctions de pics de Stembridge ($q=1$) qui sont les expansions quasisymétriques naturelles des fonctions symétriques de Schur et $Q$ Schur. Dans cet article, nous montrons que nos fonctions $q$-fondamentales fournissent une expansion quasisymétrique des fonctions symétriques Hall-Littlewood $S$ avec paramètre $t=-q$.}


\usepackage[backend=bibtex]{biblatex}
\begin{document}

\maketitle

\section{Introduction}
We define the $q$-fundamental quasisymmetric functions as the $q$-deformed generating functions for enriched $P$-partitions on labelled chains \cite{GriVas22, GriVas23}. These functions naturally interpolate between I. Gessel's fundamental (\cite{Ges84}, $q=0$) and J. Stembridge's peak (\cite{Ste97}, $q=1$) quasisymmetric functions and exhibit most of the nice properties of these two classical families. In particular, when $q$ is not a complex root of unity they span the ring of quasisymmetric functions ($\QSym$). When $q$ is a root of unity, a subfamily of our $q$-fundamentals is the basis of the algebra of extended peaks \cite{GriVas23}, a proper subalgebra of $\QSym$ that coincides with Stembridge's algebra of peaks when $q=1$. Fundamental and peak functions indexed by standard Young tableaux of shape $\la$ are respectively the quasisymmetric expansions of Schur and Schur's $Q$-symmetric functions indexed by $\la$. Finding the analogous families of symmetric functions for general $q$ appears as a natural question. We find out that $q$-fundamental functions provide a similar quasisymmetric expansion of the family $\left (S_\la(X;t)\right)_\la$, the Hall-Littlewood $S$-symmetric functions with parameter $t=-q$. After recalling the required definitions, we state and prove our main result. Finally, we look at some important consequences regarding the quasisymmetric extension of the classical homorphism between $\Lambda$, the algebra of symmetric functions and the subalgebra of $\Lambda$ spanned by Hall-Littlewood functions as well as some Cauchy-like formulas for the $S_\la(X;t)$'s.
\subsection{Integer partitions, Young tableaux and permutation statistics}

Let $\PP$ be the set of positive integers and $\PP^{\pm}$ be the set of positive and negative integers ordered by $-1<1<-2<2<-3<3<\dots$. (Our below results would actually be true for \textbf{any} total order on $\PP^{\pm}$, but we have chosen this one by force of habit.)
We embed $\PP$ into $\PP^{\pm}$ and let $-\PP \subseteq \PP^{\pm}$ be the set of all $-n$ for $n \in \PP$. For $n \in \PP$ write $[n] = \{1,\dots, n\}$ and $\mathfrak{S}_n$ the symmetric group on $[n]$.

A \emph{partition} $\la$ of an integer $n$, denoted $\la \vdash n$, is a sequence $\la=(\la_1,\la_2,\dots,\la_p)$ of $\ell(\la)=p$ parts sorted in decreasing order such that $|\la| = \sum_i{\la_i} = n$. We denote the one-part partition $(n)$ simply by $n$. A partition $\la$ is represented as a~Young diagram of $n=|\la|$ boxes arranged in $\ell(\la)$ left-justified rows so that the $i$-th row from the top contains $\la_i$ boxes. Given a second partition $\mu$ with $\ell(\mu) \leq \ell(\la)$ such that $\mu_i \leq \la_i$ (for all $i \leq \ell(\mu)$), delete the $\mu_i$ leftmost boxes of the $i$-th row to get the diagram of \emph{shape} $\la/\mu$.

A Young diagram whose boxes are filled with positive integers such that the entries are weakly increasing along the rows and strictly increasing down the columns is called a \emph{semistandard Young tableau}. If the entries are $1,2,\ldots,n$ (each appearing exactly once), we call it a \emph{standard Young tableau}. We denote $\SYT(\la/\mu)$ (resp. $\SSYT(\la/\mu)$) the set of standard (resp. semistandard) Young tableaux of shape $\la/\mu$. A \emph{marked semistandard Young tableau} is a Young diagram filled with integers in $\PP^{\pm}$ such that the entries are weakly increasing along rows and columns and such that each row contains at most once each negative integer and that each column contains at most once each positive integer. 
\begin{figure} [h]
$$
T_1 = \begin{matrix} 
\resizebox{!}{2.5cm}{%
\begin{tikzpicture}[node distance=0 cm,outer sep = 0pt]
	      \node[bsqg] (1) at (0,  0) {};
	      \node[bsqg] (2) [below = of 1] {};   
	      \node[bsqg] (3) [right = of 1] {};
	      \node[bsq] (4) [right = of 3] {\bf \Large 3};	      
	      \node[bsq] (5) [below = of 2] {\bf \Large 1};   
	      \node[bsq] (6) [right = of 2] {\bf \Large 1};      
	      \node[bsq] (7) [right = of 6] {\bf \Large 5};     
	      \node[bsq] (8) [right = of 4] {\bf \Large 3};     
	      \node[bsq] (9) [right = of 8] {\bf \Large 3};   
	      \node[bsq] (10) [right = of 7] {\bf \Large 10}; 	       	      
	      \node[bsq] (11) [below = of 5] {\bf \Large 6};   
	      \node[bsq] (12) [right = of 5] {\bf \Large 2};  	      
	      \node[bsq] (13) [below = of 11] {\bf \Large 12};  
	      \node[bsq] (14) [right = of 9] {\bf \Large 12};   
\end{tikzpicture}
}
\end{matrix}
~~~
T_2 = \begin{matrix} 
\resizebox{!}{2.5cm}{%
\begin{tikzpicture}[node distance=0 cm,outer sep = 0pt]
	      \node[bsqg] (1) at (0,  0) {};
	      \node[bsqg] (2) [below = of 1] {};   
	      \node[bsqg] (3) [right = of 1] {};
	      \node[bsq] (4) [right = of 3] {\bf \Large -4};	      
	      \node[bsq] (5) [below = of 2] {\bf \Large 2};   
	      \node[bsq] (6) [right = of 2] {\bf \Large -3};      
	      \node[bsq] (7) [right = of 6] {\bf \Large -9};     
	      \node[bsq] (8) [right = of 4] {\bf \Large 4};     
	      \node[bsq] (9) [right = of 8] {\bf \Large 4};   
	      \node[bsq] (10) [right = of 7] {\bf \Large 9}; 	       	      
	      \node[bsq] (11) [below = of 5] {\bf \Large -9};   
	      \node[bsq] (12) [right = of 5] {\bf \Large -3};  	      
	      \node[bsq] (13) [below = of 11] {\bf \Large 18};  
	      \node[bsq] (14) [right = of 9] {\bf \Large 18};   
\end{tikzpicture}
}
\end{matrix}
~~~
T_3 = \begin{matrix} 
\resizebox{!}{2.5cm}{%
\begin{tikzpicture}[node distance=0 cm,outer sep = 0pt]
	      \node[bsqg] (1) at (0,  0) {};
	      \node[bsqg] (2) [below = of 1] {};   
	      \node[bsqg] (3) [right = of 1] {};
	      \node[bsq] (4) [right = of 3] {\bf \Large 4};	      
	      \node[bsq] (5) [below = of 2] {\bf \Large 1};   
	      \node[bsq] (6) [right = of 2] {\bf \Large 2};      
	      \node[bsq] (7) [right = of 6] {\bf \Large 7};     
	      \node[bsq] (8) [right = of 4] {\bf \Large 5};     
	      \node[bsq] (9) [right = of 8] {\bf \Large 6};   
	      \node[bsq] (10) [right = of 7] {\bf \Large 9}; 	       	      
	      \node[bsq] (11) [below = of 5] {\bf \Large 8};   
	      \node[bsq] (12) [right = of 5] {\bf \Large 3};  	      
	      \node[bsq] (13) [below = of 11] {\bf \Large 10};  
	      \node[bsq] (14) [right = of 9] {\bf \Large 11};   
\end{tikzpicture}
}
\end{matrix}
$$
\caption{A semistandard tableau $T_1$, a marked semistandard tableau $T_2$, and a standard tableau $T_3$ of shape $(6,4,2,1,1)/(2,1)$. The descent set of $T_3$ is $\{2,6,7,9\}$.
The tableau $T_2$ has $\nega(T_2)= 5$ negative entries.}
\label{fig : YT}
\end{figure}

Denote $\SSYT^\pm(\la/\mu)$ the set of marked semistandard Young tableaux of shape $\la/\mu$. Define the \emph{descent set} of a standard Young tableau $T$ as $\Des(T) = \{1\leq i \leq n-1\mid i $ is in a strictly higher row than $i+1\}$.
Finally, denote the  number of negative entries of a marked tableau $T$ as $\nega(T)$.
See Figure \ref{fig : YT} for some examples.

Similarly, the descent set and the peak set of a permutation $\pi$ in $\mathfrak{S}_n$ are the subsets of $[n-1]$ defined as $\Des(\pi) = \{1\leq i \leq n-1\mid \pi(i)>\pi(i+1)\}$ and $\Peak(\pi) = \{2\leq i\leq n-1| \pi(i-1)<\pi(i)>\pi(i+1)\}$.
Finally, the {\it Robinson-Schensted (RS) correspondence} (\cite{Sch61, Sta01}) is a bijection between permutations $\pi$ in $\mathfrak{S}_n$ and ordered pairs of standard Young tableaux $(P,Q)$ of the same shape $\la \vdash n$. This bijection is descent-preserving in the sense that $\Des(\pi) = \Des(Q)$ and $\Des(\pi^{-1}) = \Des(P)$.
\subsection{Hall-Littlewood symmetric functions}
Consider the set of indeterminates $X = \left\{x_1,x_2,x_3,\ldots\right\}$. Let $\Lambda$ denote the ring of symmetric functions over
$\CC$. We use notations consistent with \cite{Mac99}. Namely, for $\la \vdash n$, let $m_{\lambda}(X)$, $h_{\la}(X)$, $e_{\la}(X)$, $p_{\la}(X)$ and $s_{\lambda}(X)$ denote the \emph{monomial}, \emph{complete homogeneous}, \emph{elementary}, \emph{power sum} and \emph{Schur} symmetric functions over $X$ indexed by $\la$. Fix a parameter $t \in \CC$ and define $q_n(X;t) \in \Lambda$ for all $n \in \ZZ$ as follows: Set $q_0(X;t) = 1$ and $q_n(X;t) = 0$ for all $n < 0$. For all $n > 0$, define $q_n(X; t)$ in finitely many indeterminates by
\begin{equation}
q_n(x_1, x_2, \ldots, x_k ;t) = \left(1-t \right)\sum_{i}x_i^n\prod_{j \neq i}\frac{x_i - tx_j}{x_i - x_j} ;
\end{equation}
then define $q_n(X; t) \in \Lambda$ by letting $k \to \infty$.
The generating function for the $q_n$ is
\begin{equation}
\sum_{n \geq 0}q_n(X;t)u^n = \prod_{i}\frac{1 - x_itu}{1 - x_iu}.
\label{eq : genfun-qn}
\end{equation}
The family $(q_n(X;t))_n$ generates a subalgebra of $\Lambda$ that we denote $\Lambda_t$. In particular, $\Lambda_t$ is a proper subalgebra of $\Lambda$ when $t$ is a root of unity. 
\begin{defn}[Hall-Littlewood $S$-symmetric functions]
\label{defn.S_HL}
For any skew shape $\la/\mu$, define the \emph{Hall-Littlewood $S$-symmetric function} indexed by $\la/\mu$ as
\begin{equation}
S_{\la/\mu}(X;t) = \det \left (q_{\la_i-\mu_j-i+j}(X;t) \right)_{i,j}.
\end{equation} 
\end{defn}
\noindent Setting $t=0$ in Definition \ref{defn.S_HL}, we obtain $S_{\la/\mu}(X;0) = s_{\la/\mu}(X)$.
When $t=-1$, $S_{\la/\mu}(X;-1)$ is a variant of \emph{Schur's $Q$-function} indexed by $\la/\mu$.

We end this section by defining a classical ring homomorphism.
 \begin{defn}[Homomorphism $\theta_t$]
\label{defn.morph}
Define a $\CC$-algebra homomorphism $\theta_t : \Lambda \longrightarrow \Lambda_t$ by setting
$$\theta_t(h_n)(X) = q_n(X;t) \qquad \text{for all } n \geq 0.$$
\end{defn}
\noindent In particular, 
as a consequence of Definition \ref{defn.S_HL}, we have $$\theta_t(s_{\la/\mu})(X) = S_{\la/\mu}(X;t).$$ 
It is also easy to see that $\theta_t(p_n)(X) = (1-t^n)p_n(X)$.

\subsection{Enriched $P$-partitions and $q$-deformed generating functions}
\label{sec : poset}
We recall the main definitions regarding weighted posets, enriched $P$-partitions and their $q$-deformed generating functions. See \cite{Ges84, GriVas21, GriVas22, Sta01, Ste97} for more details.   
\begin{defn}[Labelled weighted poset, \cite{GriVas21}]
A \emph{labelled weighted poset} is a triple $P = ([n],<_P,\epsilon)$ where $([n], <_P)$ is a \emph{labelled poset}, i.e., an arbitrary partial order $ <_P$ on the set $[n]$, and $\epsilon : [n]\longrightarrow \PP$ is a map (called the \emph{weight function}). If $\epsilon(i) = 1$ for all $i \in [n]$, then we denote $\epsilon$ by $1^n$ or just omit it.
\end{defn}
\noindent Each node of a labelled weighted poset is marked with its label and weight (Figure \ref{fig : poset}).
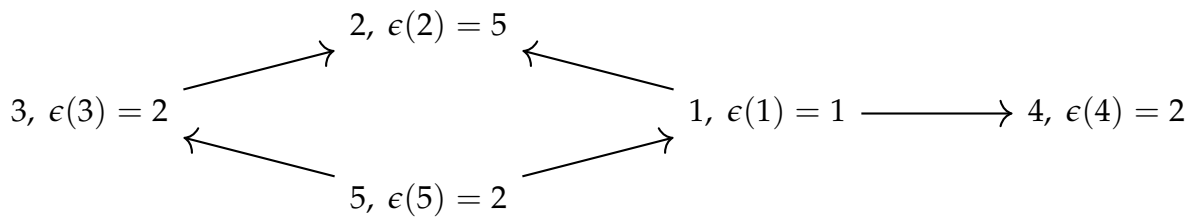
\begin{figure}[htbp]
\begin{center}
\begin{tikzcd}[row sep = small]
                                              &  & {2,\ \epsilon(2) = 5} &  &                                  &  &                       \\
{3,\ \epsilon(3) = 2} \arrow[rru, thick] &  &                                   &  & {1,\ \epsilon(1) = 1} \arrow[llu, thick] \arrow[rr, thick] &  & {4,\ \epsilon(4) = 2} \\
                                              &  & {5,\ \epsilon(5)=2} \arrow[llu, thick]  \arrow[rru, thick]   &  &                                  &  &                      
\end{tikzcd}
\end{center}
\caption{A $5$-vertex labelled weighted poset. Arrows show the covering relations.}
 \label{fig : poset}
 \end{figure}
 \begin{defn}[Enriched $P$-partition, \cite{Ste97}]\label{def : enriched}
Given a labelled weighted poset $P = ([n],<_P,\epsilon)$, an \emph{enriched $P$-partition} is a map $f: [n]\longrightarrow \PP^{\pm} $ that satisfies the two following conditions:
\begin{itemize}
\item[(i)] If $i <_P j$ and $i < j$, then $f(i) < f(j)$ or $f(i) = f(j) \in \PP$.
\item[(ii)] If $i <_P j$ and $i>j$, then $f(i) < f(j)$ or $f(i) = f(j) \in -\PP$.
\end{itemize}
We let $\mathcal{L}_{\PP^{\pm}}(P)$ denote the set of enriched $P$-partitions.
\end{defn}
Note that this set does not depend on $\epsilon$.
But the following definition brings $\epsilon$ into play:
\begin{defn}[$q$-Deformed generating function, \cite{GriVas22}]
Consider the ring $\CC \left[\left[ X \right]\right]$ of formal power series on $X$ and let $q \in \CC$ be an additional parameter. Given a labelled weighted poset $([n], <_P, \epsilon)$, define its generating function $\Gamma^{(q)}([n], <_P, \epsilon) \in \CC \left[\left[ X \right]\right]$ as
\begin{equation*}
\label{eq : weightGamma}
\Gamma^{(q)}([n], <_P, \epsilon) = \sum_{f\in\mathcal{L}_{\PP^\pm}([n],<_{P}, \epsilon)} \prod_{1\leq i\leq n}q^{[f(i)<0]}x_{|f(i)|}^{\epsilon(i)},
\end{equation*}
where $[f(i)<0] = 1$ if $f(i)<0$ and $0$ otherwise.
\end{defn}
Finally, let $X^\pm = \left\{x_{-1}, x_{1}, x_{-2},x_{2},\ldots\right\}$. In the sequel we denote $\varpi$ the substitution homomorphism $\varpi : \CC[[X^\pm]] \longrightarrow \CC[[X]]$ defined by setting $\varpi(x_i)=q^{[i<0]}x_{|i|}$ for $x_i \in X^\pm$.
\subsection{$q$-fundamental quasisymmetric functions}
We recall results from \cite{GriVas22} and \cite{GriVas23}.
\begin{defn}[$q$-Fundamental quasisymmetric functions]
Given a permutation $\pi=\pi_1\dots\pi_n$ in $\mathfrak{S}_n$, we let $P_{\pi} = ([n],<_\pi,1^n)$ be the labelled weighted poset on the set $[n]$, where the order relation $<_\pi$ is such that $\pi_i <_\pi \pi_j$ if and only if $i < j$, and where all the weights are equal to 1 (see Figure \ref{fig : monomial}).
Define the \emph{$q$-fundamental quasisymmetric function}
\begin{equation*}
L_{\pi}^{(q)} = \Gamma^{(q)}([n],<_\pi, 1^n).
\end{equation*}
\begin{figure}[htbp]
\begin{center}
\begin{tikzcd}[row sep = small]
{\pi_1} \arrow[r, thick] &
{\pi_2} \arrow[r, thick] &
\cdots\cdots\cdots \arrow[r, thick] &
{\pi_n}
\end{tikzcd}
\end{center}
\caption{The labelled weighted poset $P_{\pi}$.}
\label{fig : monomial}
\end{figure}
\end{defn}
\noindent The $q$-fundamental quasisymmetric functions belong to the subalgebra of $\CC \left[\left[ X \right]\right]$ called the ring of \emph{quasisymmetric functions ($\QSym$)}, i.e. for any strictly increasing sequence of indices $i_1 < i_2 <\cdots< i_p$ the coefficient of $x_1^{k_1}x_2^{k_2}\cdots x_p^{k_p}$ is equal to the coefficient of $x_{i_1}^{k_1}x_{i_2}^{k_2}\cdots x_{i_p}^{k_p}$. The specialisations $L_{\pi} = L_{\pi}^{(0)}$ and $K_{\pi}=L_{\pi}^{(1)}$ of $L_{\pi}^{(q)}$ are respectively the Gessel's fundamental \cite{Ges84} and Stembridge's peak \cite{Ste97} quasisymmetric functions indexed by permutation $\pi$. We have the following explicit expression.
\begin{equation}
L_{\pi}^{(q)}=\sum_{\substack{i_{1}\leq i_{2}\leq
\dots\leq i_{n};\\j\in\Peak(\pi)\Rightarrow 
i_{j-1}<i_{j+1}  }}q^{|\{j\in\Des(\pi
)|i_{j}=i_{j+1}\}|}(q+1)^{|\{i_{1},i_{2},\dots, i_{n}\}|}x_{i_{1}}x_{i_{2}}\dots x_{i_{n}}.
\label{eq : Lqq}
\end{equation}
Furthermore $q$-fundamental quasisymmetric functions admit a closed-form product and coproduct.
\begin{proposition}
Let $q \in \CC$, let $\pi \in \mathfrak{S}_n$ and $\sigma \in \mathfrak{S}_m$. The product of $L^{(q)}_{\pi}$ and $L^{(q)}_{\sigma}$  is given by
\begin{equation}
\label{eq : LL}
L^{(q)}_{\pi}L^{(q)}_{\sigma}=\sum_{\tau\in \pi \shuffle \overline{\sigma}}L^{(q)}_{\tau},
\end{equation}
where $\overline{\sigma} = n+\sigma_1\,n+\sigma_2 \dots n+\sigma_m$. Moreover, the coproduct $\Delta:\operatorname*{QSym}\rightarrow \operatorname*{QSym}\otimes\operatorname*{QSym}$ of the Hopf algebra
$\operatorname*{QSym}$ (see \cite[\S 5.1]{GriRei20}) acts on the $q$-fundamental quasisymmetric functions by
\begin{equation*}
\Delta(L^{(q)}_{\pi})
= \sum_{i=0}^n L^{(q)}_{\std(\pi_1\pi_2\dots \pi_i)} \otimes
L^{(q)}_{\std(\pi_{i+1}\pi_{i+2}\dots \pi_n)}.
\end{equation*}
Here, if $\gamma$ is a sequence of non-repeating integers, $\std(\gamma)$ is the permutation whose values are in the same relative order as the entries of $\gamma$.
\end{proposition}
According to Equation (\ref{eq : Lqq}), $L_{\pi}^{(q)}$ depends only on $n$ and on the descent set of $\pi$. Thus, we can rename $L_{\pi}^{(q)}$ as $L_{n, \Des \pi}^{(q)}$. This way, the $q$-fundamentals are now indexed by an integer $n$ and a subset of $[n-1]$. We recall two significant results. 
\begin{proposition}[\cite{GriVas22}]
\label{prop : basis} $(L_{n, I}^{(q)})_{n\geq0, I\subseteq[n-1]}$ is a basis of $\QSym$ if and only if $q \in \CC$ is not a root of unity.
\end{proposition}
\begin{proposition}[\cite{GriVas23}]\label{prop : roots} Let $p \in \PP$ and $\rho_p \in \CC$ such that $-\rho_p$ is a primitive $p+1$-th root of unity. For a subset $I$ of $[n-1]$, write $I \subseteq_p [n-1]$ if $I \cup \{0\}$ does not contain more than $p$ consecutive elements. Then $(L_{n, I}^{(\rho_p)})_{n\geq0, I\subseteq_p[n-1]}$ is a basis of a proper subalgebra $\mathcal{P}^p$ of $\QSym$. 
\end{proposition}
For general $q \in \CC$, let $\mathcal{P}^{(q)}$ denote the subalgebra of $\QSym$ spanned by the $(L_{n, I}^{(q)})_{n, I}$. If $q$ is not a root of unity then $\mathcal{P}^{(q)} = \QSym$. If $q=\rho_p$ for some $p \in \PP$ then $\mathcal{P}^{(q)} = \mathcal{P}^p$.

\section{Relating Hall-Littlewood and $q$-fundamentals functions}
The ring of symmetric functions $\Lambda$ is a subalgebra of $\QSym$ and any symmetric function may be expanded in quasisymmetric bases. The relation between Schur functions (i.e Hall-Littlewood $S$-functions with parameter $t=0$) and fundamental quasisymmetric functions is of particular interest. Let $\la /\mu$ be a skew shape. Gessel shows in \cite{Ges84} that
\begin{equation}
\label{equation.S0L}
S_{\la/\mu}(X; 0) = s_{\la/\mu}(X) = \sum_{T \in \SYT(\la/\mu)}L^{(0)}_{\Des(T)}(X).
\end{equation}
On the other hand, Stembridge shows in \cite{Ste97} that
\begin{equation}
\label{equation.S1L}
S_{\la/\mu}(X; -1) = \sum_{T \in \SYT(\la/\mu)}L^{(1)}_{\Des(T)}(X).
\end{equation}
To understand how these relations generalise for general $q$ seems to be a very legitimate question. We state our result and  some significant consequences. 
\subsection{Computing the $q$-deformed generating functions on skew diagrams}
Let $\lambda$ and $\mu$ be two partitions such that $\lambda/\mu$ is a skew shape.
Set $n = |\la| - |\mu|$.
Number the boxes of the skew Young diagram of shape $\lambda/\mu$ with the integers $1, 2, \ldots, n$ from left to right and bottom to top (as on Figure~\ref{fig.young_poset}). Define the partial order $<_{\la/\mu}$ on $[n]$ by letting $i <_{\la/\mu} j$ if and only if $i$ lies northwest of $j$ in this numbering.
Denote the labelled poset $([n], <_{\la/\mu})$ by $P_{\lambda/\mu}$.
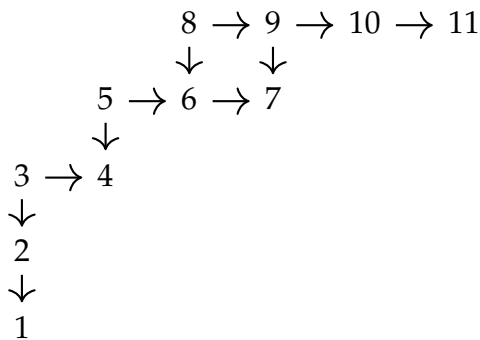
\begin{figure}[htbp]
\begin{center}
\begin{tikzcd}[row sep = small, column sep = small]
&&8\arrow[thick]{d}\arrow[thick]{r}&9\arrow[thick]{d}\arrow[thick]{r}&10\arrow[thick]{r}&11\\
&5\arrow[thick]{d}\arrow[thick]{r}&6\arrow[thick]{r}&7&&\\
3 \arrow[thick]{d} \arrow[thick]{r}&4&&&&\\
2 \arrow[thick]{d}&&&&&\\
1&&&&&
\end{tikzcd}
\end{center}
\caption{The labelled weighted poset $P_{(6,4,2,1,1)/(2,1)}$.}
\label{fig.young_poset}
\end{figure}
Then, the enriched $P_{\lambda/\mu}$-partitions are precisely the marked semistandard Young tableaux of shape $\lambda/\mu$, i.e., we have $\mathcal{L}_{\PP^{\pm}}(P_{\lambda/\mu}) = \SSYT^\pm(\la/\mu)$.
\begin{theorem}
\label{thm.SG}
Let $\lambda/\mu$ be a skew shape. Set $n = |\la| - |\mu|$. The $q$-deformed generating function of $P_{\lambda/\mu}$ is exactly the Hall-Littlewood $S$-symmetric function with parameter $t=-q$.
That is,
\begin{equation}
S_{\la/\mu}(X;-q) = \Gamma^{(q)}([n],<_{\la/\mu}).
\end{equation}
\end{theorem}
\noindent This will be proved in Section \ref{section.proof}. Using Theorem \ref{thm.SG}, we give an explicit quasisymmetric expansion of the Hall-Littlewood $S$-symmetric functions that naturally generalises (\ref{equation.S0L}) and (\ref{equation.S1L}).
\begin{theorem}
\label{thm.SL}
Let $\lambda/\mu$ be a skew shape. The Hall-Littlewood $S$-symmetric function with parameter $t=-q$ is related to $q$-fundamental quasisymmetric functions through
\begin{equation}
S_{\la/\mu}(X;-q) = \sum_{T \in \SYT(\la/\mu)}L^{(q)}_{\Des(T)}(X).
\end{equation}
\end{theorem}
\begin{proof}
Let $n = |\la|-|\mu|$. Given a marked semistandard Young tableau $T \in \SSYT^\pm(\la/\mu)$, define its \emph{standardisation} as the standard tableau $T_{0} \in \SYT(\la/\mu)$ obtained by relabelling the boxes of $T$ with the integers in $[n]$ such that: 
\begin{itemize}
\item The entries of $T$ and $T_{0}$ are in the same relative order (as long as they are distinct in $T$).
\item Identical negative entries of $T$ are relabelled from top to bottom.
\item Identical positive entries of $T$ are relabelled from left to right. 
\end{itemize}
Denote $T^{\std} = T_0$. For instance, in Figure \ref{fig : YT}, $T_1^{\std}= T_2^{\std} = T_3$. Further denote $X^{|T|} = \prod_{i \in \PP^\pm}x_{|i|}^{t_i}$ where $t_i$ is the number of entries equal to $i$ in $T$. Finally, use Theorem \ref{thm.SG} to get
\begin{align*}
S_{\la/\mu}(X;-q) &= \Gamma^{(q)}(P_{\la/\mu}) = \sum_{T \in \SSYT^\pm(\la/\mu)}q^{\nega(T)}X^{|T|}\\
&=\sum_{T_0 \in \SYT(\la/\mu)}\left (\sum_{T \in \SSYT^\pm(\la/\mu),\; T^{\std} = T_0}q^{\nega(T)}X^{|T|}\right).
\end{align*}
End the proof by noticing that the part between parentheses is exactly $L^{(q)}_{\Des(T_0)}(X)$.
\end{proof}
Recall the ring homomorphism $\theta_t$ of Definition \ref{defn.morph}. The following result is a consequence of Theorem \ref{thm.SL}. 
\begin{theorem}
\label{thm.theta}
There is a $\CC$-algebra homomorphism $\Theta_q : \QSym \longrightarrow \mathcal{P}^{(q)}$ such that for any positive integer $n$ and any subset $I\subseteq [n-1]$, we have $\Theta_q \left (L^{(0)}_{n, I} \right) = L^{(q)}_{n, I}$. Then the restriction of $\Theta_q$ to $\Lambda$ is exactly  $\theta_{-q}$ and the ring map diagram of Figure \ref{fig.ring_map} is commutative. 
\end{theorem}
\begin{figure}[htbp]
\begin{center}
\begin{tikzcd}[row sep = large, column sep = large]
\QSym \arrow{r}{\Theta_q} & \mathcal{P}^{(q)} \\
\Lambda \arrow{u} \arrow{r}{\theta_{-q}} & \Lambda_{-q} \arrow{u}
\end{tikzcd}
\end{center}
\caption{Map diagram relating $\QSym$, $\mathcal{P}^{(q)}$, $\Lambda$ and $\Lambda_{-q}$. Vertical maps are inclusion.}
 \label{fig.ring_map}
 \end{figure}
\begin{proof}
The existence and proper definition of $\Theta_q$ is a consequence of Equation (\ref{eq : LL}). To end the proof, it suffices to show that $\Theta_q(h_n)(X) = q_n(X;-q)$ for any $n \geq 0$.
But indeed, we have
\begin{align*}
\Theta_q(h_n)(X) 
&= \Theta_q(L^{(0)}_{n, \emptyset})(X)
= L^{(q)}_{n, \emptyset}(X) = S_n(X;-q)
= q_n(X;-q),
\end{align*}
where we used Theorem \ref{thm.SL} in the second-to-last equality.
This is the desired result. 
\end{proof}
\begin{remark}
Applying the morphism $\Theta_q$ to both the left-hand and right-hand sides of Equation (\ref{equation.S0L}) gives an alternative proof that $\theta_{t}\left(s_{\la/\mu}\right)(X) = S_{\la/\mu}(X;t)$. Indeed
\begin{equation*}
\Theta_q \left (s_{\la/\mu} \right)(X) = \sum_{T \in \SYT(\la/\mu)}\Theta_q \left (L^{(0)}_{n, \Des(T)} \right)(X) = \sum_{T \in \SYT(\la/\mu)}L^{(q)}_{n, \Des(T)}(X) = S_{\la/\mu}(X; -q).
\end{equation*}
\end{remark}
%
\subsection{Cauchy-like formula for Hall-Littlewood symmetric functions}
We use Theorem \ref{thm.SL} to provide an alternative proof of a classical Cauchy-like formula for Hall-Littlewood $S$-symmetric functions. Denote $Y=\{y_1, y_2, \dots \}$ an additional alphabet of commutating indeterminates independent of and commuting with $X$. Let $XY = \{x_iy_j\}_{i,j}$ be the product alphabet. We first show the following proposition.
\begin{proposition}
\label{prop.LLL}
Let $\pi \in \mathfrak{S}_n$ be a permutation. Extend the definition of $\Gamma^{(q)}$ to the alphabet $XY$ by considering $P_{\pi}$-partitions $(f,g): i \mapsto (f(i), g(i))$ with values in the set $\PP \times \PP^\pm$, which we equip with the lexicographic order. Declare a pair $(i,j) \in \PP \times \PP^\pm$ to be negative if and only if $j$ is negative. Then,
\begin{equation*}
 \Gamma^{(q)}(P_\pi)(XY) = \sum_{(f,g)\in\mathcal{L}_{\PP \times \PP^\pm}([n],<_{\pi})} \prod_{1\leq i\leq n}q^{[g(i)<0]}x_{f(i)}y_{|g(i)|}
 \end{equation*}
The $q$-fundamental indexed by $\pi$ on the product alphabet $XY$ satisfies
\begin{equation}
L^{(q)}_\pi(XY) = \Gamma^{(q)}(P_\pi)(XY) = \sum_{\sigma \circ \tau = \pi}L^{(0)}_\sigma(X)L^{(q)}_\tau(Y).
\end{equation}
\end{proposition}
\begin{proof}
The proof is similar to the one in \cite[thm 6.11]{Pet07} and not detailed here.  
\end{proof}
\noindent In \cite[III. 4. Eq. (4.7)]{Mac99}, Macdonald provides a Cauchy-like formula for Hall-Littlewood symmetric functions:
\begin{equation}
\label{equation.qsS}
q_{n}(XY;t) = \sum_{\la \vdash n}s_\la(X)S_\la(Y;t).
\end{equation}
\begin{proposition}
Equation (\ref{equation.qsS}) is a direct consequence of Proposition \ref{prop.LLL} and Theorem \ref{thm.SL}.
\end{proposition}
\begin{proof}
Fix $q \in \CC$ and use Proposition \ref{prop.LLL} to write 
\begin{equation*}
q_{n}(XY;-q) = L_{\id_n}^{(q)}(XY) = \sum_{\sigma \in \mathfrak{S}_n}L_{\sigma^{-1}}^{(0)}(X)L_{\sigma\phantom{^{-1}}}^{(q)}\!\!(Y),
\end{equation*}
where $\id_n \in \mathfrak{S}_n$ is the identity permutation. The RS correspondence allows to reindex the sum over standard Young tableaux.
\begin{align*}
q_{n}(XY;-q) &=\sum_{\la \vdash n} \sum_{T, U \in \SYT(\la)}L_{n, \Des(T)}^{(0)}(X)L_{n, \Des(U)}^{(q)}(Y)\\
&=\sum_{\la \vdash n} \left(\sum_{T \in \SYT(\la)}L_{n, \Des(T)}^{(0)}(X)\right)\left(\sum_{U \in \SYT(\la)}L_{n, \Des(U)}^{(q)}(Y)\right).
\end{align*}
Applying Theorem \ref{thm.SL} yields Equation (\ref{equation.qsS}).
\end{proof}
\section{Proof of Theorem \ref{thm.SG}}
\label{section.proof}
We now approach the proof of Theorem \ref{thm.SG}.
We generalize the setting somewhat:
Let $\prec$ be a total order on $\PP^\pm$
(for example, our order $<$, or any other total order).
Define a further binary relation $R$ on $\PP^\pm$ as follows:
For any two elements $i,j\in\mathbb{P}^{\pm}$, set%
\[
\left(  i\ R\ j\right)  \ \Longleftrightarrow\ \left(  i\preccurlyeq j\text{
but not }i=j\in-\mathbb{P}\right)  .
\]

\begin{defn}
For each non-negative integer $n$, define the formal power series
\[
H_n :=
H_{n}(X^\pm)
:=\sum_{\substack{\left(  i_{1},i_{2},\ldots,i_{n}\right)  \in\left(
\mathbb{P}^{\pm}\right)  ^{n};\\i_{1}\ R\ i_{2}\ R\ \cdots\ R\ i_{n}}}
x_{i_{1}}x_{i_{2}}\cdots x_{i_{n}}
\]
on the alphabet $X^\pm = \{x_{-1}, x_1, x_{-2}, x_2, \dots \}$.
Moreover, set $H_{n}=0$ for all $n<0$.
\end{defn}
Furthermore, for any labelled poset $([n], <_P)$, define
the new generating function
$\Gamma^{\pm}\left([n], <_P\right) \in \CC \left[\left[ X^\pm \right]\right]$
of enriched $P$-partitions by
\begin{equation*}
\Gamma^{\pm}([n], <_P) := \sum_{f\in\mathcal{L}_{\PP^\pm}([n],<_{P})} \prod_{1\leq i\leq n}x_{f(i)}.
\end{equation*}
This is a lift of our $q$-deformed generating function
$\Gamma^{(q)}([n], <_P)$.
Indeed, recall the homomorphism $\varpi : \CC \left[\left[X^\pm\right]\right] \longrightarrow \CC\left[\left[X\right]\right]$, such that $\varpi(x_i)=q^{[i<0]}x_{|i|}$ for $x_i \in X^\pm$. Clearly $$\varpi(\Gamma^{\pm}([n], <_P)) = \Gamma^{(q)}([n], <_P).$$
\begin{proposition}
Let $\la$ and $\mu$ be two partitions such that $\la/\mu$ is a skew shape. We have
\begin{equation}
\Gamma^{\pm}([n], <_{\la/\mu})
=\det\left(  H_{\lambda_{i}-\mu_{j}-i+j}\right)  _{i,j\in\left[  k\right]  }.
\label{eq.thm.Ylamu-enum.1}
\end{equation}
\end{proposition}
\begin{proof}
We want to apply \cite[\S 7]{GesVie89}. To this end, we
introduce a new relation. Let $\overline{R}$ be the complement of the binary relation $R$. (Thus,
$\overline{R}$ is the binary relation on $\mathbb{P}^{\pm}$ defined by
$\left(  i\ \overline{R}\ j\right)  \ \Longleftrightarrow\ \left(  \text{not
}i\ R\ j\right)  $.) It is easy to see that both relations $R$ and
$\overline{R}$ are transitive. Hence, the relation $R$ is semitransitive
(meaning that if $a,b,c,d\in\mathbb{P}^{\pm}$ satisfy $a\ R\ b\ R\ c$, then $a\ R\ d$ or $d\ R\ c$). Therefore, \cite[Theorem 11]{GesVie89} yields that the power series $s_{\lambda/\mu}^{R}$ (defined in \cite[\S 7]{GesVie89}) counts $R$-tableaux of shape $\lambda/\mu$. But the $R$-tableaux of shape $\lambda/\mu$ are precisely the enriched $P_{\la/\mu}$-partitions,
whereas the series $s_{\lambda/\mu}^R$ is our $\det\left(  H_{\lambda_{i}-\mu_{j}-i+j}\right)  _{i,j\in\left[  k\right]  }$.
\end{proof}
In order to prove Theorem \ref{thm.SG} from (\ref{eq.thm.Ylamu-enum.1}), we need to show that $\varpi(H_n(X^\pm)) = q_n(X; -q)$ for each $n \in \ZZ$. We proceed in three steps. First we have the following proposition, which follows easily from the definition of $H_n$.
\begin{proposition}
\label{prop.H-as-sum}Let $n \in \ZZ$. Then,
\[
H_{n}=\sum_{k=0}^{n}\ \ \sum_{\substack{U\text{ is a size-}k\\\text{subset of
}-\mathbb{P}}}\ \ \sum_{\substack{V\text{ is a size-}\left(  n-k\right)
\\\text{multisubset of }\mathbb{P}}}\left(  \prod_{u\in U}x_{u}\right)
\left(  \prod_{v\in V}x_{v}\right)
\]
(where the product over $v\in V$ takes each element with its multiplicity). In
particular, $H_{n}$ does not depend on the order $\prec$.
\end{proposition}
Secondly, we express $q_n(X;-q)$ in terms of elementary and complete homogeneous symmetric functions. 
\begin{lem}
\label{lem.qeh}
Let $n \in \ZZ$ and $q \in \CC$. Then,
\begin{equation}
q_{n}(X;-q) =\sum_{k=0}^{n}q^{k}e_{k}h_{n-k}.
\label{eq.prop.Qn.fmls.4}%
\end{equation}
\end{lem}
\begin{proof}
From \eqref{eq : genfun-qn}, we obtain
\begin{align*}
\sum_{n}q_{n}(X;-q)u^{n}=\prod_{i\geq1}\dfrac{1+qx_{i}u}{1-x_{i}u}  &  =\underbrace{\left(
\prod_{i\geq1}\left(  1+qx_{i}u\right)  \right)  }_{=\sum_{n}q^{n}e_{n}t^{n}}\underbrace{\left(  \prod_{i\geq1}\dfrac{1}{1-x_{i}u}\right)
}_{=\sum_{n}h_{n}u^{n}}\\
&  =\left(  \sum_{n}q^{n}e_{n}u^{n}\right)  \left(  \sum
_{n}h_{n}u^{n}\right) =\sum_{n}\left(  \sum_{k=0}^{n}q^{k}e_{k}%
h_{n-k}\right)u^{n}.
\end{align*}
Extracting coefficients in $u^n$ on both sides yields the desired result.
\end{proof}
Finally, use Proposition \ref{prop.H-as-sum} and Lemma \ref{lem.qeh} to relate $H_n$ and $q_n$.  
\begin{proposition}
\label{prop.Q=piH}Let $n\in\mathbb{Z}$. Then,%
\[
\varpi\left(  H_{n}(X^\pm)\right)  =q_{n}(X;-q).
\]
\end{proposition}

\begin{proof}
Applying the map $\varpi$ to both sides of Proposition \ref{prop.H-as-sum}, we obtain%
\begin{align}
\varpi\left(  H_{n}\right)   
&  =\sum_{k=0}^{n}
\ \ \sum_{\substack{U\text{ is a size-}k\\\text{subset of
}-\mathbb{P}}}
\ \ \sum_{\substack{V\text{ is a size-}\left(  n-k\right)
\\\text{multisubset of }\mathbb{P}}}
\left(  \prod_{u\in U}\varpi\left(  x_{u}\right) \right)
\left(  \prod_{v\in V}\varpi\left(  x_{v}\right) \right)
\label{pf.prop.Q=piH.1}
\end{align}
(since $\varpi$ is a continuous $\kk$-algebra homomorphism).
However, each size-$k$ subset $U$ of $-\PP$ satisfies
\[
\prod_{u\in U}\underbrace{\varpi
\left(  x_{u}\right)  }_{\substack{=qx_{-u}\\
\text{(since }u\in-\mathbb{P}\text{)}}}
= \prod_{u\in U}\left(qx_{-u}\right) = q^k\prod_{u \in U} x_{-u}
\qquad \text{(since $\left\vert U\right\vert =k$),}
\]
whereas each size-$\left(n-k\right)$-multisubset $V$ of $\PP$
satisfies
\[
\prod_{v\in V}\underbrace{\varpi\left(  x_{v}\right)
}_{\substack{=x_{v}\\\text{(since }v\in\mathbb{P}\text{)}}}
= \prod_{v\in V} x_v.
\]
In light of these two equalities, we can rewrite
\eqref{pf.prop.Q=piH.1} as
\begin{align*}
\varpi\left(  H_{n}\right)   
&  =\sum_{k=0}^{n}\ \ \sum_{\substack{U\text{ is a size-}k\\\text{subset of
}-\mathbb{P}}}\ \ \sum_{\substack{V\text{ is a size-}\left(  n-k\right)
\\\text{multisubset of }\mathbb{P}}}q^{k}\left(  \prod_{u\in U}x_{-u}\right)
\left(  \prod_{v\in V}x_{v}\right) \\
&  =\sum_{k=0}^{n}q^{k}\underbrace{\left(  \sum_{\substack{U\text{ is a
size-}k\\\text{subset of }-\mathbb{P}}}\ \ \prod_{u\in U}x_{-u}\right)
}_{\substack{=\sum_{\substack{U\text{ is a size-}k\\\text{subset of
}\mathbb{P}}}\ \ \prod_{u\in U}x_{u}\\=e_{k}}}\underbrace{\left(
\sum_{\substack{V\text{ is a size-}\left(  n-k\right)  \\\text{multisubset of
}\mathbb{P}}}\ \ \prod_{v\in V}x_{v}\right)  }_{=h_{n-k}}
=\sum_{k=0}^{n}q^{k}e_{k}h_{n-k} .
\end{align*}
By Equation (\ref{eq.prop.Qn.fmls.4}), we can rewrite this
as $\varpi\left(H_n\right) = q_{n}(X;-q)$.
\end{proof}

\printbibliography

\end{document}